%% file: classes_for_non-compact.tex
\documentclass[twoside,psamsfonts,a4paper]{amsart}

\include{preamble}

\title[Characteristic classes on Hilbert schemes]{Characteristic classes of the Hilbert schemes of points
on non-compact simply-connected surfaces}
\author{Marc A.~Nieper-Wi\ss kirchen}
\address{Max-Planck-Insitut for Mathematics, Vivatsgasse 7, 53111 Bonn, Germany}
\email{marc@nieper-wisskirchen.de}
\thanks{The author would like to thank the Max Planck Institute for Mathematics
in Bonn for its hospitality and support during the preparation of this paper.}
\thanks{The theorems in this paper supersede the results of the author's
earlier paper ``Equivariant cohomology, symmetric functions and the Hilbert scheme of points on the total space of the
inverstible sheaf $\sheaf O_{\set P^1}(-2)$ over the projective line''.}

\date{\today}

\begin{document}

    \begin{abstract}
    	We prove a closed formula expressing any multiplicative characteristic class
    	evaluated on the tangent bundle of the Hilbert schemes of points
    	on a non-compact simply-connected surface.
    	
    	As a corollary, we deduce a closed formula for the Chern character of the tangent
    	bundles of these Hilbert schemes.
    	
    	We also give a closed formula for the multiplicative characteristic classes
    	of the tautological bundles associated to a line bundle on the surface.

        We finally remark which implications the results here have for the Hilbert schemes of points of
        an arbitrary surface.
    \end{abstract}

	\maketitle

    \tableofcontents
    
    \section{Introduction}
    
    Let $X$ be a quasi-projective connected smooth surface over the complex numbers
    (\emph{surface} for short).
    We denote by $\Hilb n X$ the Hilbert scheme of $n$ points on $X$,
    parametrising zero-dimensional subschemes of $X$ of length
    $n$. It is a quasi-projective variety (\cite{grothendieck:hilbert})
    and smooth of dimension $2n$ (\cite{fogarty}).
    
    The direct sum $\bigoplus_{n \ge 0} H^*(\Hilb n X, \set Q[2n])$ of the (shifted)
    cohomology spaces of all Hilbert schemes carries a natural grading given by the cohomological
    degree and a weighting given by the number of points $n$. Likewise, the symmetric algebra
    $S^*(t^{-1} H^*(X, \set Q[2])[t^{-1}])$ carries a grading by cohomological degree and
    a weighting. The weighting is defined such that $t^{-n} H^*(X, \set Q[2])$ is of pure weight $n$. 
    
    There is a natural vector space isomorphism
    \[
        S^*(t^{-1} H^*(X, \set Q[2])[t^{-1}]) \to \bigoplus_{n \ge 0} H^*(\Hilb n X, \set Q[2n])
    \]
    respecting the grading and weighting (\cite{grojnowski}, \cite{nakajima}). The image of $1$
    under the isomorphism is denoted by $\vac$ and called the \emph{vacuum}.
    For each $l > 0$ and $\alpha \in H^*(X, \set Q[2])$, the linear operator
    on $\bigoplus_{n \ge 0} H^*(\Hilb n X, \set Q)$ that corresponds via this isomorphism
    to multiplying by $t^{-l} \alpha$ is denoted by $q_l(\alpha)$.
    These operators are called \emph{creation operators}.
    
    One open problem in the study of Hilbert schemes is to express any characteristic class
    of $\Hilb n X$ by a closed formula in terms of the creation operators applied on the
    vacuum.
    
    The following general result holds (\cite{boissiere:tangent}, \cite{boissiere-nieper:universal}):
    For each multiplicative characteristic
    class $\phi$ with values in a $\set Q$-algebra~(\cite{hirzebruch:methods}),
    there are unique series $(a_\phi^\lambda)_\lambda$,
    $(b_\phi^\lambda)_\lambda$, $(c_\phi^\lambda)_\lambda$, $(d_\phi^\lambda)_\lambda$
    of elements in $A$, indexed by the set $\Set \lambda$ of all
    partitions, such that
    \begin{equation}
        \label{eq:universal_formula}
    	\sum_{n \ge 0} \phi(\sheaf T_{\Hilb n X}) = \exp\sum_\lambda \left(
    	    a_\phi^\lambda q_\lambda(1) + b_\phi^\lambda q_\lambda(K_X)
    	    + c_\phi^\lambda q_\lambda(e_X) + d_\phi^\lambda q_\lambda(K^2_X)
    	\right) \vac
    \end{equation}
    holds for all smooth surfaces $X$
    (where we have to view the equation in the completion
    $\prod_{n \ge 0} H^*(\Hilb n X, \set Q[2])$ with respect to
    the weighting).
    
    Here, the operator $q_\lambda(\alpha)$ for an $r$-tuple $\lambda = (\lambda^1, \cdots,
    \lambda^r)$ (e.g.~a partition of length $r$)
    and a class $\alpha \in H^*(X, \set Q[2])$ is defined as follows:
    Let $\delta^r\colon X \to X^r$ be the diagonal embedding and
    $\delta^r_!\colon H^*(X, \set Q[2]) \to H^*(X, \set Q[2])^{\otimes r}$ the induced
    proper push-forward map (which is of degree $2r - 2$). We then set
    $q_\lambda(\alpha) := \sum q_{\lambda_1}(\alpha_{(1)}) \cdots q_{\lambda_r}(\alpha_{(r)})$,
    where $\delta^r_! \alpha = \sum \alpha_{(1)} \otimes \cdots \otimes \alpha_{(r)}$ (in Sweedler notation).
    
    It remains to express the coefficients $(a_\phi^\lambda), \ldots, (d_\phi^\lambda)$
    by a closed formula depending on $\phi$. In this paper, we completely solve the problem
    for all non-compact simply-connected surfaces.
    
    We also give a similar result for the \emph{tautological bundles $\Hilb n {\sheaf F}$}
    on the Hilbert schemes of $X$ associated to a line bundle $\sheaf F$ on $X$ (\cite{lehn:tautological}).
    The tautological bundle $\Hilb n {\sheaf F}$ is a vector bundle of rank $n$ on $\Hilb n X$, whose fibre over
    a point $\xi$ in $\Hilb n X$ (i.e.~a zero-dimensional subscheme of length $n$ on $X$) is given by
    $H^0(X, \sheaf F \otimes \sheaf O_\xi)$.
    
    By universality of the formulas for the characteristic classes (\cite{boissiere-nieper:universal},
    \cite{boissiere-nieper:generating}),
    the results found here also apply to compact or non-simply-connected surfaces. For an arbitrary surface,
    however, this gives only partial results.
    
    Let us finally remark that we haven't made any assumptions on the canonical divisor of $X$. This is
    noteworthy insofar as many results for Hilbert schemes have been obtained in closed form only for $K_X = 0$.
    
    \section{The results}
    
    From now on, let $\phi$ be a fixed multiplicative characteristic class with values in the
    $\set Q$-algebra $A$. It is given by a power series $f \in 1 + x A[[x]]$ such that
    $\phi = \prod_i f(x_i)$, where the $x_i$ are the Chern roots (\cite{hirzebruch:methods}).

    Let $g \in xA[[x]]$ be the compositional inverse of the power series
    $G := \frac x{f(x) f(-x)}$.

    The first main result of this paper is the following:
    \begin{theorem}
        \label{thm:tangent}
        Let $X$ be non-compact simply-connected surface.
        The multiplicative class $\phi$ evaluated on the tangent bundle
        of the Hilbert schemes of points on $X$ is given by
        \[
            \sum_{n \ge 0} \phi(\sheaf T_{\Hilb n X}) = \exp\left(
              \sum_{k \ge 1} (a_k q_k(1) + b_k q_k(K_X))
              + \sum_{k, l \ge 1} a_{k, l} q_{(k, l)}(1)
            \right) \vac,
        \]
        where the $A$-valued sequences $(a_k)_k$, $(b_k)_k$, and $(a_{k, l})_{k, l}$
        are defined by
        \[
            \sum_{k \ge 1} k a_k x^k = g(x),\quad \sum_{k \ge 1} b_k = \frac 1 2 \log \frac {f^2(-g(x))} {g'(x)},
        \]
        and
        \[
            \sum_{k, l \ge 1} a_{k, l} x^k y^l = \frac 1 2 \log
            \frac{(x - y) f(g(x) - g(y)) f(g(y) - g(x))}{g(x) - g(y)}.
        \]
    \end{theorem}
    The theorem is proven in the following section.

    From this theorem, we can deduce a formula for the Chern character.
    \begin{corollary}
        The Chern character of the tangent bundle of the Hilbert schemes of points on $X$ is
        given by
    	\[
    	    \sum_{n \ge 0} \ch(\sheaf T_{\Hilb n X}) = \left(
              \sum_{k \ge 1} (a_k q_k(1) + b_k q_k(K_X))
              + \sum_{k, l \ge 1} a_{k, l} q_{(k, l)}(1)
            \right) \exp(q_1(1)) \vac,
        \]
        where the rational sequences $(a_k)_k$, $(b_k)_k$, and $(a_{k, l})_{k, l}$
        are defined by
        \[
            \sum_{k \ge 1} a_k x^k = \sum_{m \ge 0} \frac 2 {(2m + 1)!} \frac{x^{2m + 1}}{2m + 1},\quad
            \sum_{k \ge 1} b_k x^k = \sum_{m \ge 0} \frac {-1} {(2m + 1)!} (x^{2m + 1} + x^{2m + 2})
        \]
        and
        \[
            \sum_{k, l \ge 1} a_{k, l} x^k x^l = \sum_{m \ge 1} \frac 1{(2m)!}
            \sum_{k + l = 2m} \left((-1)^k \binom{2m}{k} - 1\right) x^k y^l.
        \]
    \end{corollary}
    
    \begin{proof}
        Set $\widetilde\ch := \ch - \ch_0$. Let $\phi$ be the
        multiplicative class with values in the ring $\set Q[\epsilon]$ of dual numbers that
        corresponds to the power series
        $f(x) := 1 + \epsilon (\exp x - 1)$. Then $\widetilde \ch = [\epsilon] \phi$.
        Here, we use the notation $[t^n] F(t)$ to denote the $t^n$-coefficient
        of a polynomial $F(t)$.
        
        For the choice of $\phi$ at hand,
        the power series $G$ is given by $G(x) = x - 2 \epsilon x (\cosh x - 1)$.
        Its compositional inverse is $g(x) = x + 2 \epsilon x (\cosh x - 1)$. 
        Write
        \[
            \sum_{n \ge 0} \phi(\sheaf T_{\Hilb n X}) = \exp\left(
                \sum_{k \ge 1} (\tilde a_k q_k(1) + \tilde b_k q_k (K_X))
                + \sum_{k, l \ge 1} \tilde a_{k, l} q_{(k, l})(1)\right) \vac.
        \]
        By Theorem~\ref{thm:tangent}, we have
        \begin{equation}
            \sum_{k \ge 1} k \tilde a_k x^k = x + 2 \epsilon x (\cosh x - 1),\quad
            \sum_{k \ge 1} \tilde b_k x^k = - \epsilon (1 + x) \sinh x
        \end{equation}
        and
        \[
            \sum_{k, l \ge 1} \tilde a_{k, l} x^k y^l
            = \epsilon \left(\cosh(x - y) - \frac{x \cosh x - y \cosh y}{x - y}\right).
        \]
        Because of $\widetilde \ch = [\epsilon] \phi$ and the fact that
        $(\ch - \widetilde \ch)(\sheaf T_{\Hilb n X}) = 2n \frac{q_1^n(1)}{n!} \vac$,
        the corollary follows.
    \end{proof}

    By the same methods we will use in the proof of the preceding theorem, one can also prove our second theorem,
    which is about the tautological bundles.
    
    Let $h \in xA[[x]]$ be the compositional inverse of the power series $H := \frac x {f(-x)}$.
    \begin{theorem}
        Let $X$ be a non-compact simply-connected surface and $\sheaf F$ a line bundle on $X$ with first Chern class
        $F$. 
        The multiplicative class $\phi$ evaluated on the tautological bundles $\Hilb n {\sheaf F}$
        of the Hilbert schemes of points on $X$ is given by
        \[
            \sum_{n \ge 0} \phi(\Hilb n {\sheaf F}) = \exp\left(
              \sum_{k \ge 1} (a_k q_k(1) + b_k q_k(K_X) + c_k q_k(F))
              + \sum_{k, l \ge 1} a_{k, l} q_{(k, l)}(1)
            \right) \vac,
        \]
        where the $A$-valued sequences $(a_k)_k$, $(b_k)_k$, $(c_k)_k$, and $(a_{k, l})_{k, l}$
        are defined by
        \[
            \sum_{k \ge 1} k a_k x^k = h(x),\quad \sum_{k \ge 1} b_k = \frac 1 2 \log \frac{h^2(x)}{x^2 h'(x)},
            \quad
            c_k = \log \frac x {h(x)}
        \]
        and
        \[
            \sum_{k, l \ge 1} a_{k, l} x^k y^l = \frac 1 2 \log
            \frac{(x - y)h(x) h(y)}{x y (h(x) - h(y))}.
        \]
    \end{theorem}
    The proof of this Theorem is omitted as it is very similar to the one of Theorem~\ref{thm:tangent}.

    \begin{remark}
    	In fact, our theorems hold for all surfaces for which $H^4(X, \set Q) = 0$, and for which $H^1(X, \set Q) = 0$
    	or $H^3(X, \set Q) = 0$. In particular, this condition is fulfilled for $X$ being non-compact
    	and simply-connected.
    \end{remark}
    
    \section{The proof}

    By our assumptions on $X$, it is $e_X = 0 = K_X^2$. Furthermore,
    it is $\delta^r_! 1 = 0$ for $r \ge 3$ and $\delta^r_! K_X = 0$ for $r \ge 2$.
    The general formula~\eqref{eq:universal_formula} thus
    specialises to the following:
    There are unique $A$-valued sequences $(a_k)_k$, $(b_k)_k$, and $(a_{k, l})_{k, l}$
    with $a_{k, l} = a_{l, k}$ such that
    \[
        \sum_{n \ge 0} \phi(\sheaf T_{\Hilb n X}) = \exp\left(
            \sum_{k \ge 1} (a_k q_k(1) + b_k q_k(K_X))
            + \sum_{k, l \ge 1} a_{k, l} q_{(k, l)}(1)
        \right) \vac
    \]
    for all non-compact simply-connected surfaces $X$.
    (The uniqueness can be proven as in the general case;
    see~\cite{boissiere-nieper:universal}.)
    
    By degree reasons, it follows that
    \[
        \sum_{n \ge 0} \phi_{n - 1}(\sheaf T_{\Hilb n X}) = \left(\sum_{k \ge 1} a_k q_k(1)\right)
        \vac
    \]
    and
    \begin{equation}
    	\label{eq:main}
        \sum_{n \ge 0} \phi_{n}(T_{\Hilb n X}) = \exp\left(
            \sum_{k \ge 1} b_k q_k(K_X)
            + \sum_{k, l \ge 1} a_{k, l} q_{(k, l)}(1)
        \right) \vac
   \end{equation} 
    where $\phi_n$ denotes the component of $\phi$ of cohomological degree $2n$. The results
    in~\cite{boissiere-nieper:universal} already give the claimed values of the $a_k$
    in Theorem~\ref{thm:tangent}.
    
    It remains to prove the claims about the values of the $b_k$ and the $a_{k, l}$.
    
    The idea of the proof
    is to explicitely calculate the left hand side of~\eqref{eq:main}
    when $X$ is the total space
    of a line bundle $\sheaf O_{\set P^1}(-\gamma)$ over $\set P^1$ by means of equivariant
    cohomology. Here, $\gamma > 1$. By comparing coefficients, we will be able to
    deduce the generating series for the $a_k$, $b_k$, and $a_{k, l}$.
    
    We note the following cohomological facts about $X$:
    The cohomological fundamental class $h$ of a fibre
    of the line bundle $\sheaf O_{\set P^1}(-\gamma)$
    spans $H^2(X, \set Q)$. The canonical class is given by
    $K_X = (\gamma - 2) h$ and it is $\delta^2_! 1 = - \gamma (h \otimes h)$.
    
    On $X$, we fix the $\set C^\units$-action that is described in~\cite{li-qin-wang:jack}.
    It has two isolated fixpoints. This $\set C^\units$-action induces a
    $\set C^\units$-action on the Hilbert scheme $\Hilb n X$, also with isolated fixpoints.
    To each pair $(\lambda^0, \lambda^1)$ of partitions with $\abs \lambda^0 + \abs \lambda^1 = n$, i.e.~to each
    bipartition of $n$, corresponds exactly one fixpoint $\xi_{\lambda^0, \lambda^1}$ (\cite{li-qin-wang:jack}).
    
    The torus $\set C^\units$ acts on the tangent space at the fix points
    $\xi_{\lambda^0, \lambda^1}$ with weights $W_{\lambda^0}(-1, -1)$
    and $W_{\lambda^1}(\gamma - 1, 1)$, where $W_\lambda(\alpha, \beta)$ is a multiset 
    defined by
    \[
        W_\lambda(\alpha, \beta) := \Set{
            \alpha l(w) + \beta (a(w) + 1), - \alpha (l(w) + 1) - \beta a(w) \mid w \in D_\lambda
        }
    \]
    for a partition $\lambda$. In the definition, $D_\lambda$ denotes the Young diagram of $\lambda$,
    and $l(w)$ ($a(w)$) denotes the leg (arm) length of a cell $w$ in the Young diagram. For the notion of
    leg and arm length, see~\cite{macdonald}; the description of the weights is
    from~\cite{li-qin-wang:jack}.
    
    As $\Hilb n X$ has no cohomology in odd degrees and is thus equivariantly formal,
    we can apply the localisation formula in equivariant
    cohomology (\cite{atiyah-bott}, \cite{edidin-graham}).
    For the equivariant characteristic class $\phi^{\set C^\units}$,
    the localisation formula gives
    \begin{equation}
        \label{eq:localisation}
        u^n \phi^{\set C^\units}_n(\sheaf T_{\Hilb n X})
        = \sum_{\substack{\lambda^0, \lambda^1
            \\ \abs{\lambda^0} + \abs{\lambda^1} = n}}
            [\xi_{\lambda^0, \lambda^1}]_{\set C^\units}
            [u^n]
            \prod_{w \in W_{\lambda^0} \amalg W_{\lambda^1}} \frac {f(w u)}{w}    	
    \end{equation}
    in $H^*_{\set C^\units}(\Hilb n X, \set Q)$, which is a $H^*(B \set C^\units, \set Q) = \set Q[u]$-module.
    Here $[\xi_{\lambda^0, \lambda^1}]_{\set C^\units}$ is the
    equivariant cohomological fundamental class of the fixpoint $\xi_{\lambda^0, \lambda^1}$.
    
    In~\cite{li-qin-wang:jack}, the authors introduce equivariant cohomology classes
    $[\lambda^0, \lambda^1] \in H^{2n}_{\set C}(\Hilb n X, \set Q)$ for which
    \[
        \frac{u^n [\lambda^0, \lambda^1]}{c'_{\lambda^0}(-1, -1) c'_{\lambda^1}(\gamma - 1, 1)}
        = \frac{[\xi_{\lambda^0, \lambda^1}]_{\set C^\units}}
        {\prod_{w \in W_{\lambda^0} \amalg W_{\lambda^1}} w}
    \]
    holds. Here $c'_{\lambda}(\alpha, \beta) := \prod_{w \in D_\lambda}(\alpha l(w) + \beta (a(w) + 1)$ for
    a partition $\lambda$.
    
    Let $\Lambda$ be the ring of symmetric functions over the rationals in the variables $x_1, x_2, \ldots$.
    The power symmetric functions are denoted by $p_n := \sum_i x_i^n$.
    Let $\psi\colon \Lambda \otimes \Lambda \to \bigoplus_{n \ge 0} H^*(\Hilb n X, \set Q)$
    be the linear map that maps $(p_{\lambda_1} \cdots p_{\lambda_r})
    \otimes (p_{\mu_1} \cdots p_{\mu_s})$ to $q_{\lambda_1}(h) \cdots q_{\lambda_r}(h)
    q_{\mu_1}(h) \dots q_{\mu_s}(h) \vac$.
    
    Let us denote by $j^*\colon H^*_{\set C^\units}(\Hilb n X, \set Q) \to H^*(\Hilb n X, \set Q)$
    the natural map from equivariant to ordinary cohomology. In~\cite{li-qin-wang:jack} it is
    proven that
    $j^*([\lambda^0, \lambda^1]) = \psi(P^{1}_{\lambda^0}
    \otimes P^{(\gamma - 1)^{-1}}_{\lambda^1})$. Here
    $P^{\alpha}_\lambda \in \Lambda$ is the Jack polynomial to the parameter
    $\alpha$ (\cite{macdonald}).
    
    As $j^*$ maps the equivariant characteristic class $\phi^{\set C^\units}$
    to the ordinary characteristic class $\phi$,equation~\eqref{eq:localisation} thus yields
    \begin{equation}
    	\label{eq:phi_n}
        \phi_n(\sheaf T_{\Hilb n X}) = \sum_{\substack{\lambda^0, \lambda^1
            \\ \abs{\lambda^0} + \abs{\lambda^1} = n}}
            \psi(P^{1}_{\lambda^0}
                \otimes P^{(\gamma - 1)^{-1}}_{\lambda^1})
            \frac{[u^n] \prod_{w \in W_{\lambda^0} \amalg W_{\lambda^1}} f(wu)}
            {c'_{\lambda^0}(-1, -1) c'_{\lambda^1}(\gamma - 1, 1)}
    \end{equation}
    
    Let us define a linear map $\rho\colon H^{2n}(\Hilb n X, \set Q) \to A[x, y]$ that
    maps $q_{\lambda_1}(h) \dots q_{\lambda_r}(h)\vac$ to
    $(x^{\lambda_1} + y^{\lambda_1}) \cdots (x^{\lambda_r} + y^{\lambda_r})$ for each partition
    $(\lambda_1, \ldots, \lambda_r)$ of $n$.
    
    For an element $f_0 \otimes f_1 \in \Lambda \otimes \Lambda$ of total degree $n$, we have
    $\rho(\psi(f_0 \otimes f_1)) = f_0(x, y) f_1(x, y)$. Here, for a
    symmetric function $f \in \Lambda$, the expression $f(x, y)$ means to substitute $x_1$ by $x$,
    $x_2$ by $y$ and $x_i$ for $i \ge 3$ by $0$.
    
    We apply the map $\rho$ on both sides of~\eqref{eq:phi_n}. In view of~\eqref{eq:main}, this yields
    \begin{equation}
        \label{eq:def_w}
        \begin{aligned}
        \sum_{n \ge 0} \rho(\phi_n(T_{\Hilb n X}))
        & = \exp\left(\sum_{k \ge 1} (\gamma - 2) b_k (x^k + y^k) - \gamma \sum_{k, l \ge 1} a_{k, l}
            (x^k + y^k) (x^l + y^l)\right)
        \\
        & = \sum_{\lambda^0, \lambda^1}
           P^1_{\lambda^0}(x, y)
               \otimes P^{(\gamma - 1)^{-1}}_{\lambda^1}(x, y)
           \frac{[u^{\abs{\lambda^0} + \abs{\lambda^1}}]
           \prod_{w \in W_{\lambda^0} \amalg W_{\lambda^1}} f(wu)}
           {c'_{\lambda^0}(-1, -1) c'_{\lambda^1}(\gamma - 1, 1)}
        \\
        & =: Z_\gamma(x, y)        	
        \end{aligned}
    \end{equation}
    
    From this, we can read off
    \begin{align*}
    	\sum_{k, l \ge 1} a_{k, l} x^k y^l
    	& = - \frac 1 4 (\log Z_2(x, y) - \log Z_2(x, 0) - \log Z_2(0, y)),
        \\
        \intertext{and}
    	\sum_{k \ge 1} b_k x^k & = \log Z_3(x, 0) - \frac 3 2 \log Z_2(x, 0).
    \end{align*}

    By Lemma~\ref{lem:z2} and Lemma~\ref{lem:z3}, Theorem~\ref{thm:tangent} is proven.
    
    \section{Two combinatorial equalities}
    
    The following
    two Lemmas give the values for the power series $Z_2(x, y)$ and $Z_3(x, 0)$. Both were needed in the proof
    of Theorem~\ref{thm:tangent} in the previous section.
    
    \begin{lemma}
        \label{lem:z2}
        The value of $Z_2(x, y)$ in $A[[x, y]]$ is given by
    	\begin{equation*}
    		\begin{aligned}
        	    Z_2(x, y) & = g'(x) g'(y) \left(\frac{G(g(x) - g(y)))}{x - y}\right)^2
        	    \\
        	    & = g'(x) g'(y) \left(\frac{g(x) - g(y)}{(x - y)f(g(x) - g(y)) f(g(y) - g(x))}\right)^2.			
    		\end{aligned}
    	\end{equation*}
    \end{lemma}
    
    \begin{proof}
    	For $\gamma = 2$, the Jack polynomials $P_\lambda^{(\gamma - 1)^{-1}} = P^1_\lambda$
    	are exactly the Schur polynomials $s_\lambda$ (\cite{macdonald}). It is
    	$s_{\lambda}(x, y) = 0$ for any partition $\lambda$ of length greater than two
    	and $s_{(a, b)}(x, y) = \frac{x^{a + 1} y^b - y^{a + 1} x^b}{x - y}$
        for $a \ge b \ge 0$.
    	In particular only the summands corresponding to partitions of length two or less contribute to the
    	sum $Z_2(x, y)$. For these partitions, we have
        \[
    	    W_{(a, b)}(1, 1)
    	    = W_{(a, b)}(-1, -1)
    	    = \Set{
    	        \pm 1, \ldots, \pm b, \pm 1, \ldots, \widehat{\pm (a - b + 1)}, \ldots,
    	        \pm (a + 1)
    	    },
    	\]
    	$c'_{(a, b)}(1, 1) = \frac{(a + 1)! b!}{a - b + 1}$,
    	and $c_{(a, b)}(-1, -1) = (-1)^{a + b} c'_{(a, b)}(1, 1)$.
    	Set $F(x) := f(x) f(-x)$. The defining equation~\eqref{eq:def_w} takes on the following form:	
    	\begin{equation*}
    	    \begin{aligned}
        	    Z_2(x, y) & = \frac 1 {(x - y)^2}
        	    \sum_{\substack{a \ge b \ge 0 \\ c \ge d \ge 0}}^\infty
        	    (-1)^{a + b} \frac{(a - b + 1)(c - d + 1)}{(a + 1)!b!(c + 1)!d!}
        	    \\
        	    & \cdot (x^{a + 1} y^b - y^{a + 1} x^b)(x^{c + 1}y^d - y^{c + 1} x^d)
        	    \\
        	    & \cdot [u^{a + b + c + d}] \frac{
        		    \prod_{w = -b}^{a + 1} F(wu) \prod_{w = -d}^{c + 1} F(wu)}
        		  {F((a - b + 1) u)F((c - d) + 1)u)}
        		\\
        		& = - \frac 1 {(x - y)^2}
        	    \sum_{\substack{a > b \ge 0 \\ c > d \ge 0}}^\infty
        	    (-1)^{a + b} \frac{(a - b)(c - d)}{a!b!c!d!}
        	    \cdot (x^{a} y^b - y^{a} x^b)(x^{c}y^d - y^{c} x^d)
        	    \\
        	    & \cdot [u^{a + b + c + d - 2}] \frac{
        		    \prod_{w = -b}^{a} F(wu) \prod_{w = -d}^{c} F(wu)}
        		  {F((a - b) u)F((c - d))u)}
        		\\
        		& = - \frac 1{(x - y)^2} \sum_{\substack{a, b \ge 0 \\ c, d \ge 0}}^\infty
        	    (-1)^{a + b} \frac{(a - b)(c - d)}{a!b!c!d!}
        	    \cdot x^{a + c} y^{b + d}
        	    \\
        	    & \cdot [u^{a + b + c + d - 2}] \frac{
        		    \prod_{w = -b}^{a} F(wu) \prod_{w = -d}^{c} F(wu)}
        		  {F((a - b) u)F((c - d))u)}
        	\end{aligned}
        \end{equation*}
        (The first equality is a simple index shift for the summation variables $a$ and $c$. The second equality
        has been obtained by writing the term symmetrically in $a$ and $b$, and $c$ and $d$.)
        We introduce new summation variables $r = a + c$ and $s = b + d$. This and the fact that $F$ is an even
        power series yields
        \begin{equation*}
        	\begin{aligned}
        	    Z_2(x, y)
        		& = - \frac 1{(x - y)^2} \sum_{r, s \ge 0} x^r y^s [u^{r + s}]
        		\sum_{a = 0}^r \sum_{b = 0}^s (-1)^{a + b}
        		\frac{u^2 (a - b)(r + s - (a - b))}{a! b! (r - a)! (s - b)!}
        		\\
        		& \cdot \underbrace{\frac{\prod_{w = a - r}^a F(wu) \prod_{w = b - s}^b F(wu)}
        		{F((a - b) u) F((r + s - (a + b)) u)}}_{(*)}.
    	    \end{aligned}  
    	\end{equation*}
    	By Lemma~\ref{lem:dots}, stated below, the term $(*)$ is equal to
    	$\frac{F^{r + 1}(a u) F^{s + 1}(b u)}{F((a - b) u) F((b - a) u)} + X(u, a, b)$, where
        $X(u, a, b) \in u A[[a u, b u, u]]$. As
        \begin{equation}
            \label{eq:cases}
            \sum_{a = 0}^r \frac{(-1)^a a^k}{a! (r - a)!} = \begin{cases}
            	0 & \text{for $k < r$ and} \\
            	(-1)^r & \text{for $k = r$}
            \end{cases}        	
        \end{equation}
        (\cite{boissiere-nieper:universal}), the term $X(u, a, b)$ cannot contribute and
        we thus have
        \begin{equation*}
            Z_2(x, y) = - \frac 1 {(x - y)^2} \sum_{r + s \ge 0} (-1)^{r + s} x^r y^s
            [a^r b^s] \frac{F^{r + 1}(a) F^{s + 1}(b) (a - b) (b - a)}{F(a - b) F(b - a)}.
        \end{equation*}
        Recall that we defined $G(z) = \frac z {F(z)}$. With this definition, we have
        \[
            Z_2(x, y) = - \frac 1 {(x - y)^2} \sum_{r, s = 0}^\infty (-x)^r (-y)^s
            \res_{(a, b)} \frac{G(a - b) G(b - a)}{G^{r + 1}(a) G^{s + 1}}.
        \]
        By the Lagrange--Good formula (\cite{krattenthaler}), this is equivalent to
        \[
            Z_2(-G(x), -G(y)) = \frac 1 {G'(x) G'(y)} \frac{G(x - y)}{G(x) - G(y)}
            \frac{G(y - x)}{G(y) - G(x)}.
        \]
        As $G$ is an odd power series, and $g$ is defined as the inverse power series of $G$,
        the claim of the Lemma follows.
    \end{proof}

    \begin{lemma}
        \label{lem:z3}
        The value of $Z_3(x, 0)$ in $A[[x]]$ is given by
    	\[
    	    Z_3(x, 0) = f(-g(x)) g'(x).
    	\]
    \end{lemma}

    \begin{proof}
    	The proof goes along the same lines as the proof of Lemma~\ref{lem:z2}.
    	By definition
    	of the Jack polynomials it is
    	$P_\lambda^\alpha(x, 0) = 0$ for any partition of length greater than $1$
    	and $P_{(a)}^\alpha(x) = x^a$ for $a \ge 0$ (independently of the parameter $\alpha$).
    	Thus only the summands corresponding to partitions of
    	length one contribute to the sum $Z_3(x, 0)$. Note that
    	$W_{(a)}(-1, -1) = \Set{\pm 1, \ldots, \pm a}$,
    	$W_{(a)}(2, 1) = \Set{- (a + 1), \ldots, -2, 1, \ldots, a}$, $c'_{(a)}(-1, -1) = (-1)^a a!$,
    	and $c'_{(a)}(2, 1) = a!$.
    	As above, we set $F(x) := f(x) f(-x)$. The definition of $Z_3(x, 0)$ in~\eqref{eq:def_w} yields
    	\[
    	    Z_3(x, 0) = \sum_{n \ge 0} \sum_{a\ge 0} (-1)^a x^n [u^n]
    	    \frac{f(- (n - a + 1) u) \prod_{w=a-n}^{a} F(w u)}{a! (n - a)! f(-u)}.
    	\]
    	By Lemma~\ref{lem:dots}, there exists a power series $X(a, u) \in uA[[au, u]]$ with
    	\begin{equation*}
        	Z_3(x, 0) = \sum_{n \ge 0} \sum_{a \ge 0} (-1)^a x^n [u^n] \left(
        	\frac{f(a u) F^{n + 1}(au)}{a! (n - a)!} + X(a, u)\right).
        \end{equation*}
        By~\eqref{eq:cases}, it follows that
        \[
            Z_3(x, 0) = \sum_{n \ge 0} (-1)^n x^n [a^n] f(a) F^{n + 1}(a)
            = \sum_{n \ge 0} (-x)^n \res_a \frac{f(a)}{G^{n + 1}(a)}.
        \]
    	Again we can make use of the Lagrange--Good formula, which yields
    	\[
    	    Z_3(- G(x), 0) = \frac{f(x)}{G'(x)}.
    	\]
    	As $G$ is an odd power series, the claim of the Lemma follows.
    \end{proof}

    \begin{lemma}
    	\label{lem:dots}
    	Let $f \in 1 + xA[[x]]$ be any formal power series with constant coefficient $1$. Let $n$ and $a$
    	be two integers with $n \ge a \ge 0$.
        Then there exists a power series $X(a, u) \in uA[[ua, u]]$ in each of which
        monomials the degree in $a$ is strictly less than the degree in $u$ such
        that
        \[
            \prod_{w = a - n}^a f(wu) = f^{n + 1}(a u) + X(a, u).
        \]
    \end{lemma}

    \begin{proof}
        This statement has already appeared at the end of~\cite{boissiere-nieper:universal},
        where also a proof can be found.
    \end{proof}

    \appendix
    
    \bibliographystyle{amsalpha}
    \bibliography{my}
\end{document}

%% file: preamble.tex
\usepackage{ucs}
\usepackage[utf8x]{inputenc}

\usepackage{amsmath}
\usepackage{amscd}
\usepackage{amsxtra}
\usepackage{amssymb}
\usepackage{amsthm}
\usepackage{stmaryrd}

\newtheorem{theorem}{Theorem}[section]
\newtheorem{corollary}[theorem]{Corollary}
\newtheorem{lemma}[theorem]{Lemma}

\theoremstyle{definition}

\theoremstyle{remark}

\newtheorem{remark}[theorem]{Remark}

\DeclareMathOperator{\res}{res}
\newcommand{\abs}[1]{|#1|}

\newcommand{\ch}{\operatorname{ch}}

\newcommand{\Hilb}[2]{{#2}^{[#1]}}
\newcommand{\Set}[1]{\left\{#1\right\}}

\newcommand{\sheaf}[1]{\mathcal #1}

\newcommand{\set}[1]{\mathbf{#1}}

\newcommand{\units}{\times}
\newcommand{\state}[1]{|#1\rangle}
\newcommand{\vac}{\state{0}}